\documentclass[11pt,a4paper]{amsart}

\usepackage{amsthm, amsfonts, amssymb, amsmath, latexsym, enumerate}
\usepackage[all]{xy}
\usepackage[latin1]{inputenc}
\usepackage{mathrsfs}
\usepackage{array}
\newtheorem{theorem}{Theorem}[section]
\newtheorem{lemma}[theorem]{Lemma}
\newtheorem{claim}[theorem]{Claim}
\newtheorem{corollary}[theorem]{Corollary}
\newtheorem{proposition}[theorem]{Proposition}
\newtheorem{conjecture}[theorem]{Conjecture}

\theoremstyle{definition}
\newtheorem{definition}[theorem]{Definition}
\newtheorem{remark}[theorem]{Remark}
\newtheorem{example}[theorem]{Example}

\newcommand{\II}{\mathscr{I}}

\newcommand{\OO}{\mathscr{O}}

\newcommand{\SSS}{\mathcal{S}}
\newcommand{\OOO}{\mathcal{O}}

\newcommand{\AAA}{\mathcal{A}}

\DeclareMathOperator{\rk}{rk}

\newcommand{\C}{\mathbb C}
\newcommand{\G}{\mathbb G}
\newcommand{\N}{\mathbb N}
\newcommand{\p}{\mathbb P}

\newcommand{\R}{\mathbb R}
\newcommand{\Z}{\mathbb Z}
\newcommand{\Ha}{\mathbb H}
\newcommand{\Oa}{\mathbb O}

\newcommand{\map}{\dasharrow}

\SelectTips{cm}{} \newdir{ >}{{}*!/-5pt/@{>}}
\numberwithin{equation}{section}
\begin{document}

\title[On the dimension of secant varieties]
{On the dimension of secant varieties}

\author{Luca Chiantini}
\address{Luca Chiantini\\
Dipartimento di Scienze Matematiche e Informatiche\\
Pian Dei Mantellini 44, I--53100 Siena, Italy}
\urladdr{\tt{http://www2.dsmi.unisi.it/newsito/docente.php?id=4}}
\email{{\tt chiantini@unisi.it}}

\author{Ciro Ciliberto}
\address{Ciro Ciliberto\\ Dipartimento di Matematica\\
Universit\`a di Roma Tor Vergata\\ Via della Ricerca Scientifica\\
 00133 Roma, Italia} \urladdr{\tt{http://www.mat.uniroma2.it/~cilibert/}}
\email{{\tt cilibert@axp.mat.uniroma2.it}}


\subjclass[2000]{Higher secant varieties, tangential projections,
special varieties}

\sloppy

\begin{abstract} In this paper we
generalize Zak's theorems on tangencies and
on linear normality as well as Zak's definition and
classification of Severi varieties. In particular we
find sharp lower bounds for the dimension of higher secant
varieties of a given variety $X$ under suitable regularity assumption
on $X$, and we classify varieties for which the bound is attained.
\end{abstract}

\maketitle
\tableofcontents

\section*{Introduction}

Let $X\subset \p^ r$ be an irreducible, projective, non--degenerate variety of
dimension $n$. For any non--negative integer $k$
one can consider the {\it $k$--secant variety} of $X$,
which is the Zariski closure in $\p^r$ of the union of all
$k$--dimensional subspaces of $\p^r$ that are spanned by $k+1$
independent points of $X$. Secant varieties are basic projective invariants
related to a given variety $X$ and their understanding is of
primary importance in the study of the geometry and topology of $X$.
As such, they have been,  since more than a century,  the object
of important research in algebraic geometry.  For instance, the
classification of {\it defective varieties}, i.e. the ones for
which some secant variety has dimension smaller than the expected,
goes back
to several classical authors, like Terracini \cite{Terr1},
Palatini \cite{Pal1}, \cite{Pal},  and Scorza \cite{Scorza}, \cite {Scorza2}
to mention a few.  For recent developments on this classical theme
see \cite {WDV}, \cite {ChCi}, \cite  {chcil} and the basic reference \cite {Zak}.

In more recent times the interest in the geometry of
projective varieties has been revived
by Zak's epochal work (see \cite {Zak}). Specifically, Zak first
proved his so--called
theorem on tangencies, a basic tool
which, although very classical in spirit,
completely escaped the consideration of
the classics. This theorem was used by Zak
to prove a sharp lower bound
for the dimension of the first secant variety
to a smooth variety $X$, as well as the classification of those
varieties achieving the bound,
i.e. the so--called  \emph{Severi varieties}.

In this paper we present an extension of these
results of Zak's. Namely we first extend the theorem on tangencies,
then we provide,
under suitable regularity assumptions for a variety $X$,
a lower bound for the dimension of its higher secant varieties, and
finally we classify the varieties for which the bound is attained.

To be specific, we introduce in \S \ref {sec:jktan} the notion of
$J_k$--tangency extending the concept of $J$--tangency which
is one of the cornerstones of Zak's theorem on tangencies.  The notion of
$J_k$--tangency is crucial for us, so we devote to it, and
to related concepts, all  \S \ref {sec:jktan}.
In \S \ref {sec:exttan} we prove Theorem \ref {Zakk}, which is
the announced generalization of the theorem on tangencies.
In \S \ref {sec:ln} we prove our extension of Zak's theorem on
linear normality, i.e. Theorem \ref {thm:extln}  providing a
sharp lower bound for the dimension of  the
$k$--secant variety to varieties having a suitable tangential behaviour which we call
\emph{$R_k$--property}, where $R$ stays for
\emph{regularity} (see Definition \ref{def:wr}). Basic tools
in the proof are the generalized theorem on tangencies, as well
as a few basic fact about secant varieties, defects and
contact loci, that we present in \S\S \ref {sec:jsd} and
\ref {sec:hyer}.  

Notice that, without suitable regularity assumptions, 
it looks quite unlikely to get
good bounds for the dimension of higher secant varieties.
Examples, together with a nice account of the general 
theory, can be found in \cite{MSSC},
where several partial results are given.

In \S \ref {sec:class} we define \emph {$k$--Severi varieties}
as the irreducible $R_k$--varieties for which the bound
in the extended theorem on linear normality is attained. Smoothness is not
required in the definition. However we prove that $k$--Severi varieties
are smooth (see Theorem \ref  {thm:ksev}).
The classification
of $k$--Severi varieties is given in Theorem \ref {thm:zakclassthm}.
The main point here is to observe that $k$--Severi varieties are \emph{Scorza
varieties} in the sense of Zak (see \cite {Zak}, Chapter VI). Then our
classification theorem follows from Zak's classification of Scorza varieties
in \cite {Zak}. However, a crucial point here is
the smoothness of certain contact
loci (ensured by Lemma \ref {cor:all}  and Claims \ref  {claim:apt} and \ref {claim:alf}
in the proof of Theorem \ref {thm:ksev}), which is essential in Zak's analysis of Scorza varieties. It is
well known that strong motivations for Zak's work have been
Hartshorne's conjectures. One of them, i.e. Hartshorne's Conjecture \ref {conj:hartb}
on linear normality, has been proved by Zak. The other (see Conjecture
\ref {conj:harta}) is still unsolved. In \S \ref {sec:spec} we speculate on a possible
extension of this conjecture which may be suggested by the
results of the present paper.

We want to finish by observing that, besides the intrinsic interest of the subject,
defective varieties,  or more generally
properties of secant varieties, are relevant also in other fields
of mathematics, such as expressions of polynomials as sums of powers,
Waring type problems, polynomial interpolation,  rank tensor
computations and canonical forms, Bayesian networks, algebraic
statistics etc. (see \cite{Ci} as a general reference, \cite{CGG},
\cite{GSS}, \cite{IK}, \cite{RS}).  This classical subject is therefore still very lively
and widely open to future research.

\section{Preliminaries and notation}

We work over the field $\C$ of complex
numbers and we consider the projective space $\p^r = \p^r_{\C}$,
equipped with the tautological line bundle $\OO_{\p^r}(1)$.

If $Y\subset \p^r$ is a subset, we denote by $\langle Y \rangle$
the span of $Y$. We say that $Y$ is {\it non--degenerate} if $\langle Y \rangle=\p^r$.
A linear subspace of dimension $n$ of $\p^ r$
will be called a \emph{$n$--subspace} of $\p^ r$.

Given a subscheme $X\subset \p^r$, $I_X$ will denote its
homogeneous ideal and $\II_X$ the ideal sheaf of $X$.

Let $X\subseteq\p^r$ be a scheme.
By a {\it general point} of $X$ we mean
a point varying in some dense open Zariski subset of
some irreducible component of $X$. We will denote by $\dim(X)$ the maximum
of the dimensions of the irreducible components of $X$.
We will often assume that
$X$ is {\it pure}, i.e. all the irreducible components of $X$ have
the same dimension.   If $X$ is projective, reduced and pure,
we will say it is a \emph{variety}.

Let $X\subset \p^ r$ be a variety. We will denote by ${\rm Sing}(X)$ the
closed Zariski subset of singular points of $X$
Let $x\in X-{\rm Sing}(X)$ be a smooth point. We
will denote by $T_{X,x}$ the embedded tangent space to $X$ at $x$,
which is a $n$--subspace of $\p^ r$. More generally, if $x_1,\ldots,x_k$
are smooth points of $X$, we will set

$$T_{X,x_1,\ldots,x_k}=\langle \bigcup_{i=1}^ n T_{X,x_i} \rangle.$$

We will denote by $V_{n,d}$ the \emph{$d$--Veronese variety}
of $\p^ n$, i.e. the image of $\p^ n$ via the
$d$--Veronese embedding

$$v_{r,d}: \p^ r\to \p^ r, \quad r={{r+d}\choose d}-1.$$

Given positive integers $0<m_1\leq\ldots \le m_h$ we will
denote by  ${\rm Seg}(m_1,\ldots ,m_h)$ the {\it Segre variety} of type
$(m_1,\ldots,m_h)$, i.e. the image of  $\p^{m_1}\times \ldots \times
\p^{m_h}$ in $\p^r$, $r=(m_1+1)\ldots (m_h+1)-1$, via
the  {\it Segre embedding}

$$s_{m_1,\ldots ,m_h}: \p^{m_1}\times \ldots \times
\p^{m_h}\to \p^ r.$$

Let $0\leq a_1\leq a_1\leq \ldots\leq a_n$ be
integers and set $\p(a_1,\ldots,a_n):=\p({\mathcal
O}_{\p^1}(a_1)\oplus\ldots\oplus{\mathcal O}_{\p^1}(a_n))$.
Set $r=a_1+\ldots+a_n+n-1$ and consider the morphism

$$\phi_{a_1,\ldots,a_n}:\p(a_1,\ldots,a_n)\to \p^r$$

\noindent defined by the sections of the
line bundle ${\mathcal O}_{\p(a_1,\ldots,a_n)}(1)$. We denote the
image of $\phi_{a_1,\ldots,a_n}$  by $S(a_1,\ldots,a_n)$. As soon as
$a_n>0$, the morphism $\phi_{a_1,\ldots,a_n}$ is birational to its image. Then the
dimension of $S(a_1,\ldots,a_n)$ is $n$, its degree is
$a_1+\ldots+a_n=r-n+1$ and $S(a_1,\ldots,a_n)$ is a {\it rational
normal scroll}, which is smooth if and only if $a_1>0$.

We will denote by $\G(m,n)$ the \emph{Grassmann variety}
of $m$--subspaces in $\p^ n$, embedded in $\p^r$, $r= { {n+1}\choose {m+1}}-1$,
via the \emph {Pl\"ucker embedding}.

For all integers $k\geq 1$, we will denote by $S_k$ the
\emph{$k$--spinor variety}, which
parametrizes the family of $(k-1)$--subspaces contained in a smooth quadric
of dimension $2k-1$. The variety $S_k$ is smooth, of dimension ${k+1}\choose 2$,
its Picard group is generated by a very ample divisor which
embeds $S_k$ in $\p^{2^ k-1}$.

\section{Joins, secant varieties and defects}\label{sec:jsd}

Let $X_0,\ldots,X_k$ be varieties in $\p^ r$.
The \emph{join} $J(X_0,\ldots,X_k)$
of $X_0,\ldots,X_k$ is the closure
in $\p^ r$ of the set\medskip

\centerline{$\{ x\in \p^r: x$  lies in the span of $k+1$
independent points $p_i\in X_i, 0\leq i\leq k\}$.} \medskip

\noindent We will use the exponential notation $J(X_1^ {m_1},\ldots,X_h^ {m_h})$
if $X_i$ is repeated $m_i$ times, $1\leq i\leq h$.
If $X_0,\ldots,X_k$ are irreducible, their
join is also irreducible.  The definition is  independent of the order
of $X_0,\ldots,X_k$ and one has

$$\dim(J(X_0,\ldots,X_k))\leq  \min\{r, k+\sum_{i=0}^ k \dim(X_i)\}.$$
The right hand side is called the \emph{expected dimension} of the join.

If $X$ is irreducible of dimension $n$, we will set $S^k(X)=J(X^ {k+1})$, and we will
call $S^{k}(X)$ the  $k$--{\it secant variety} of $X$. This is an irreducible
variety of dimension

\begin{equation} \label{defect} s^{(k)}(X):= \dim (S^k(X))\leq
 \min\{r,n(k+1)+k\}:=e^{(k)}(X).\end{equation}
 Again, the right hand side is called the {\it expected
dimension} of $S^{k}(X)$.

If $X$ is reducible of dimension $n$, then $J(X^ {k+1})$ is in general reducible
and not pure. In this case, we consider
the union of all joins $J(X_0,\ldots, X_k)$, where
$X_0,\ldots, X_k$ are distinct irreducible components of $X$.
It is convenient for us to denote this by $S^ k(X)$
and call it the $k$--{\it secant variety} of $X$.  With this convention, formula
 \eqref {defect} still holds. The varieties $X$ we will be considering
 next, even if reducible, will have the property that $S^ k(X)$ is pure.
 We will therefore often make this assumption.

One says that $X$ is {\it $k$--defective} when strict inequality
holds in (\ref{defect}). One calls
$$\delta_k(X):=e^{(k)}(X)-s^{(k)}(X)$$
the $k$--{\it secant defect} of $X$.
There is however a slightly different
concept of $k$--defect, which will be useful for  us,
i.e.  the concept of  {\it $k$--fiber defect} $f_k(X)$,
defined as

\begin{equation} \label{fiberdefect}
f_k(X) = (k+1)n+k-s^{(k)}(X).
\end{equation}
Notice that $f_k(X)=\delta_k(X)$, if $r\geq (k+1)n+k$, while
otherwise $f_k(X)=\delta_k(X)+(k+1)n+k-r$, thus $f_k(X)$ can be positive
even if $\delta_k(X)=0$.

\begin{remark} \label{rem:fdef} The reason for the name \emph{fibre defect}
is the following.  Assume $S^ k(X)$  pure.
Then $f_k(X)$ equals the dimension of
the family of $(k+1)$--secant $k$--spaces to $X$ passing through a
general point of $S^{k}(X)$.

Indeed, consider the \emph{abstract
secant variety} $\SSS^k(X)$ which is the union of
the closures of the sets
$$
\{ (p_0,\ldots,p_k,x)\in X_0\times\ldots\times X_k\times  \p^r:  \\
x\in\langle p_0,\ldots,p_k\rangle\simeq \p^ k\}
$$
with $X_0,\ldots,X_k$ distinct irreducible components of $X$. Then
the image of the projection $p: \SSS^ k(X)\to \p^ r$ is $S^ k(X)$ and
$f_k(X)$ is the dimension of a general fiber of $p$.
Hence one may have $f_k(X)>0$ even if $X$ is not
$k$-defective: this happens when $S^{k}(X)=\p^r$
and $r<(k+1)n+k$.
\end{remark}

We will use abbreviated notation like $s^{(k)}, e^{(k)}, \delta_k, f_k$
instead of $s^{(k)}(X), e^{(k)}(X), \delta_k(X), f_k(X)$ if there
is no danger of confusion. Also, we may drop the index $k$ when $k=1$.

\section{Secant varieties and contact loci}\label {sec:hyer}

If $X_0,\ldots, X_k$ are projective varieties in $\p^r$, then Terracini's Lemma
describes the tangent space to their join at a general point of it
(see \cite{Terr1} or, for modern versions, \cite{Adl}, \cite{WDV}, \cite {order},
\cite{Dale1}, \cite{Zak}).

\begin{theorem}\label{terracini1}  Let $X_0,\ldots, X_k$
be varieties in $\p^ r$. If $p_i\in X_i$, $0\leq i\leq k$,  are
general points  and $x\in \langle p_0,\ldots,p_k\rangle$ is a general point, then:

$$T_{J(X_0,\ldots,X_k),x}=\langle T_{X_0,p_0}, \ldots , T_{X_k,p_k}\rangle.$$

In particular, if $X\subset\p^r$ is an irreducible, projective variety, if $p_0,\ldots,p_k\in X$ are
general points and $x\in \langle p_0,\ldots,p_k\rangle$ is a general point, then:

$$T_{S^k(X),x}=T_{X,p_0,\ldots,p_k}.$$

\end{theorem}

We recall
a useful consequence of Terracini's Lemma, which is well
known in the irreducible case (see \cite {Zak}, p. 106). 
The easy proof can be left to the reader.

\begin{proposition} \label{nofill} Let $X\subset\p^r$ be a
non--degenerate variety. If $\dim(J(X^ k))=\dim(J(X^ {k+1}))$
then $J(X^ k)=\p^r$.  Similarly, if $\dim(S^ k(X))=\dim (S^ {k-1}(X))$, then
$S^k(X)=\p^ r$.
\end{proposition}

Given a variety $X\subset \p^ r$ of dimension $n$, the \emph{Gauss map} of $X$ is the rational map

$$g_X: X\dasharrow \G(n,r)$$

\noindent  defined at the smooth points of $X$ by mapping $x\in X-{\rm Sing}(X)$
to $T_{X,x}$.  It is well known that, if $x\in X$ is a general point, then the closure of the
fibre of $g_X$ through $x$ is a linear subspace $\Gamma_{X,x}$ of $\p^ r$.

\begin{definition}\label {def:gauss} In the above setting, $\Gamma_{X,x}$ is called the general \emph {Gauss fibre}
of $X$ and its dimension the \emph{tangential defect} of $X$,  denoted by $t(X)$.
We will set $t_k(X)=t(S^k(X))$.
\end{definition}

Note that, if $X$ is smooth, then $t(X)=0$ (see \cite  {Zak}).

Let  $X\subset \p^ r$ be non--degenerate variety such that $s^ {(k)}(X)<r$.
Terracini's Lemma implies that $t_k(X)\geq k$. More precise
information about $t_k(X)$ will be provided in a while.
First we are going to introduce a few remarkable families of subvarieties of $X$ related
to $S^ k(X)$.

Given $x\in S^k(X)$ a general point,
i.e. $x\in \langle p_0,\ldots,p_k\rangle$ is a general point, with $p_0,\ldots,p_k\in X$
general points, consider the Zariski closure of the set\medskip

\centerline{  $\{ p\in X- {\rm Sing}(X):  T_{X,p} \subseteq T_{S^k(X),x} \}$. }\medskip

\noindent  We will denote by $\Gamma_{p_0,\ldots,p_k} $
the union of all irreducible components of this locus containing $p_0,\ldots,p_k$,
and by $\gamma_k(X)$ its dimension, which clearly does not depend on $p_0,\ldots,p_k$.
Note that $\Gamma_{p_0,\ldots,p_i}\subseteq \Gamma_{p_0,\ldots,p_k}$
for all $i=1,\ldots,k$.
We set

$$\Pi_{p_0,\ldots,p_k}= \langle \Gamma_{p_0,\ldots,p_k} \rangle.$$

We will use the abbreviated
notation $\Gamma_k, \Pi_k, \gamma_k$ if no confusion arises.
Note that $\Pi_k$ contains $\langle p_0,\ldots,p_k\rangle$,
hence it contains $x$.

\begin{definition} In the above setting, we will call
$ \Gamma_{p_0,\ldots,p_k}$  the {\it tangential $k$--contact locus} of
$X$ at $p_0,\ldots,p_k$. We will call
$\gamma_k(X)$ the $k$-{\it tangential defect} of $X$.
\end{definition}

Let $X\subset\p^r$ be an irreducible,
projective variety as above and let again $p_0,\ldots,p_k\in X$
 be general points. Consider
the projection of $X$ with centre $T_{X,p_1,\ldots,p_k}$. We call
this a {\it general $k$--tangential projection} of $X$, and we
denote it by $\tau_{X,p_1,\ldots,p_k}$, or $\tau_{p_1,\ldots,p_k}$, or $\tau_k$.
 We denote by $X_{p_1,\ldots,p_k}$, or simply by $X_k$, its image. By Terracini's
Lemma, the map $\tau_{k}$ is generically finite to its image if
and only if $s^{(k)}(X)=s^{(k-1)}(X)+n+1$.

 \begin{definition} Let $p_0\in X$ be a general point. Let $
\Psi_{p_0,\ldots,p_k}$ be the component of the fibre of $\tau_{X,p_1,\ldots,p_k}$
containing $p_0$. We will denote it by $\Psi_k$ if no confusion arises.
It is called the \emph{projection $k$--contact locus}
of $X$ at $p_0,\ldots,p_k$ and we will denote by $\psi_k(X)$, or $\psi_k$,
its dimension, which is independent of $p_0,\ldots,p_k$.
This will be called the \emph {projection $k$--defect}.
\end {definition}

\begin{remark}\label{conctactin} Notice that  $\Gamma_{p_0,\ldots,p_k}$
contains $\Psi_{p_0,\ldots,p_k}$.
Indeed $T_{X,p_0,\ldots,p_k}$ projects, via $\tau_k$, to the tangent space of $X_k$
at the point $\tau_k(p_0)$, thus it is tangent along the component of the fiber
containing $p_0$. In particular we get $\gamma_k\geq\psi_k$. Equality holds
if and only if the Gauss map of $X_k$ is generically finite to its image,
which is equivalent to say that it is birational to its image.

One has $\Psi_{p_0,\ldots,p_i}\subseteq \Psi_{p_0,\ldots,p_k}$ for all
 $i=1,\ldots, k$.

One of the main consequences of Terracini's Lemma is
that, if $X$ is $k$--defective, then $\psi_k>0$, so that
 $\gamma_k$  is also positive.
\end{remark}

Other relevant items related to the secant variety
$S^ k(X)$ are the so called \emph
{entry loci}.

\begin{definition}\label{def:el} Let $x\in  S^ k(X)$ be a point.
We define the \emph{entry locus} $E_{k,x}$
of $x$ with respect to $X$ as the closure of the set\medskip

\centerline{$\{ z\in X:$  there is $x'\in S^ {k-1}(X): x'\neq z$ and $ x\in \langle z,x'\rangle\}$.}
 \medskip

Alternatively, consider the fibre $F_x$ of $p: \SSS^k(X)\to S^k(X)$ over $x$.
The entry locus $E_{k,x}$ is the image of $F_x$ under the projection
$p_1: \SSS^ k(X)\to X$ to the first factor.

If $x\in S^ k(X)$ is  a general point, we may denote $E_{k,x}$ simply by $E_k$.
\end{definition}

Next we can provide interesting information about the $k$--contact loci.

\begin{lemma}\label {lem:a} Let $X\subset \p^ r$ be an irreducible,
 non--degenerate variety such that $s^ {(k)}(X)<r$. If $p_0,\ldots,p_k\in X$ are general
points and  $q_0,\ldots,q_k$ are general points on $\Gamma_{p_0,\ldots,p_k}$,
such that $q_i$ specializes to $p_i$, for all $i=0,\ldots,k$.
Then $\Gamma_{p_0,\ldots,p_k}=\Gamma_{q_0,\ldots,q_k}$.
\end {lemma}

\begin{proof}  One has
$T_{X,q_i}\subset T_{S^k(X),x}=T_{X,p_0,\ldots,p_k}$, for all $i=0,\ldots,k$, thus
$T_{X,q_0,\ldots,q_k}=T_{X,p_0,\ldots,p_k}$.  This immediately implies the assertion.
\end{proof}

\begin{proposition}\label {prop:tandef} Same hypotheses as in Lemma \ref {lem:a}. Then:
\begin{itemize}

\item [(i)] $\Gamma_{p_0,\ldots,p_k}$ is smooth at $p_0,\ldots,p_k$; moreover it is
either irreducible or it consists of
$k+1$ irreducible components of the same dimension $\gamma_k$
each containing one of the points $p_0,\ldots,
p_k$ as its general point;

\item [(ii)] $f_i(\Gamma_k)=f_i(X)$ for all $i=1,\ldots,k$;

\item [(iii)] $\Pi_k=S^ k(\Gamma_k)$ equals the general Gauss fibre $\Gamma_{S^ k(X),x}$ of $S^ k(X)$, whereas  $S^ i(\Gamma_k)\neq \Pi_k$ for all $i=1,\ldots, k-1$;

\item [(iv)] $t_k(X)=\dim(\Pi_k)=k\gamma_k+k+\gamma_k-f_k(X)$.
\end{itemize}
\end{proposition}

\begin{proof} Part (i) follows by Lemma \ref {lem:a} and by monodromy on the general points $p_0,\ldots,p_k$ (see \cite {chcil}).

Let us prove part (ii).
We assume $\Gamma_k=\Gamma_{p_0,\ldots,p_k}$ irreducible, otherwise the same argument works.
Let $x\in \langle p_0,\ldots,p_k\rangle$ be a general point of $S^ k(X)$.
Let also $q_0,\ldots,q_i$ be general points on $\Gamma_k$ and let $y\in \langle q_0,\ldots,q_i\rangle $ be a general point of $S^ i(\Gamma_k)$.
By the generality assumption on
$p_0,\ldots,p_k$, also $q_0,\ldots,q_i$ are general points on $X$, hence $y$ is a general point of $S^ i(X)$.

Since  $q_0,\ldots,q_i$  are in $\Gamma_k$, we have $T_{S^i(X),y}=T_{X,q_0,\ldots,q_i}\subset
T_{S^k(X),x}$.  Moreover, we have a $f_i(X)$--dimensional family of $(i+1)$--secant  $i$--spaces to $X$ passing through $y$. Let $\langle r_0,\ldots,r_i\rangle$ be a general element of such a family, with
$r_0,\ldots,r_i\in X$.
Then $T_{X,r_0,\ldots,r_i}=T_{S^i(X),y}\subset   T_{S^k(X),x}$, which shows that $r_0,\ldots, r_i\in \Gamma_k$. This implies part (ii).

Let us prove (iii). We have $S^ k(\Gamma_{p_0,\ldots,p_k})\subseteq \Pi_{p_0,\ldots,p_k}\subseteq \Gamma_{S^ k(X),x}$. Let $y$ be a general point of $\Gamma_{S^ k(X),x}$, hence
$y\in \langle q_0,\ldots,q_k\rangle$, with $q_0,\ldots,q_k\in X$. Since
$ T_{S^k(X),y}= T_{S^k(X),x}$, we have $q_0,\ldots,q_k\in \Gamma_{p_0,\ldots,p_k}$, thus
$y\in S^ k(\Gamma_{p_0,\ldots,p_k})$. This proves the first assertion.  If $S^ i(\Gamma_k)=\Pi_k$ for some $i<k$, then we would have $S^ i(X)=S^ k(X)$, contradicting $S^ k(X)\neq \p^ r$ (see Proposition \ref {nofill}).

Part (iv) easily follows. \end{proof}

\begin{remark}\label {rem:reason} It is useful to notice that the $k$--contact loci are \emph{responsible} for the $k$--fibre defect of $X$. In fact Proposition \ref {prop:tandef} tells us that
they have the same $k$--fibre defect of $X$ and through $k+1$ general points of $X$ there is
one of them passing.

Moreover, by applying Proposition \ref {prop:tandef} for all positive integers $i<k$, one has that
$S^ i(\Gamma_i)$ is the linear subspace $\Pi_i$, which is also the general Gauss fibre of $S^ i(X)$, and
$t_i(X)=i\gamma_i+i+\gamma_i-f_i(X)$.
\end{remark}

The next proposition shows an important relation between fibre and projection defects.

\begin{proposition}\label{defects} Let $X\subset \p^ r$ be an irreducible
variety of dimension $n$, such that  $s^{(k)}(X)<r$. Then:
$$f_i= \psi_1 + \dots +\psi_i,\quad \forall\quad i \leq k.$$
The same holds if $X$ is reducible but $S^ {(k)}(X)$ is pure
and $s^{(k)}(X)<r$.
\end{proposition}

\begin{proof}
The proof is by induction on $i$. The case $i=1$ is an immediate
consequence of Terracini's Lemma.

Suppose $i>1$. Consider
the general $i$--tangential projection $\tau_i$ of $X$ from
$T_{X,p_1,\ldots,p_i}$. One has  $\dim(X_i)=n-\psi_i$
and the general tangent
space to $X_i$ is the projection of $T_{X,p_0,p_1,\ldots,p_i}$,
$p_0$ being a general point of $X$.
We have $\dim(T_{X,p_0,p_1,\ldots,p_i})=in+i+n-f_i$ and
$\dim(T_{X,p_1,\ldots,p_i})=in+i-1-f_{i-1}.$ Hence
$n-\psi_i=n-f_i+f_{i-1}$, and the assertion follows by induction.
\end{proof}

\begin{corollary}\label{cor:equal}  In the above setting,
for all $i=1,\ldots,k$, fix
$p_1,\ldots,p_i$ general points in $\Gamma_k$. Then
 the general projection $i$--contact locus of $\Gamma_k$ coincides with $\Psi_i$. In particular
 $\psi_i(\Gamma_k)=\psi_i(X)$.
\end{corollary}

\begin{proof}  Note that $T_{\Gamma_k,p_1,\ldots,p_i}\subseteq T_{X,p_1,\ldots,p_i}\cap \Pi_k$.
Moreover $T_{X,p_1,\ldots,p_i}$ does not contain $\Gamma_k$, otherwise it would contain the whole of $X$, since $p_0\in \Gamma_k$ is a general point of $X$.
Therefore it makes sense to consider the restriction  to $\Gamma_k$ of the $i$--tangential projection $\tau_{X,i}$, which factors through the $i$--tangential projection $\tau_{\Gamma_k,i}$. This implies that
$\psi_i(\Gamma_k)\leq \psi_i(X)$. By Proposition \ref {defects} and part (ii) of Proposition \ref {prop:tandef}, the equality has to hold for all
$i=1,\ldots,k$ and the assertion follows.\end{proof}

An useful information about the entry loci is provided by the following:

\begin{proposition}\label {lem:dimel} Let $X\subset \p^ r$ be an irreducible
variety with $s^ {(k)}(X)<r$. Then
$E_k$ is pure of dimension $\psi_k.$
The same holds if $X$ is reducible and $S^ i(X)$ is pure,
with $k-1\leq i\leq k$.
\end{proposition}

\begin{proof} Let $x\in S^ k(X)$ be a general point.
The fibre $F_x$ of $p: \SSS^k(X)\to S^ k(X)$ over $x$
is pure of dimension $f_k$.
The projection to the first factor yields a dominant
map $q: F_x\to E_{k,x}$. Let $z$ be a general point
in a component of $E_{k,x}$, and
let $F_{x,z}$ be the fibre of $q$ over $z$.
Let $\xi=(z,p_1,\ldots,p_k,x)$ be a general
point in a component $Z$ of  $F_{x,z}$. Note that $\langle
p_1,\ldots,p_k\rangle$ is a $(k-1)$--space intersecting the line
$\langle x,z\rangle$ at a point $p$. By Proposition \ref {nofill},
the point $p$ does not depend on $\xi$. Moreover one
sees that, for $x$ and $z$ general, then
$p$ is a general point in $S^ {k-1}(X)$. Hence we have a map

$$Z\dasharrow \SSS^{k-1}(X),\quad \xi\to (p_1,\ldots,p_k,p)$$
which is birational to its image, and this, in turn, is a component of the fibre of
$\SSS^{k-1}(X)\to S^ {k-1}(X)$ over $p$. This shows that any component $Z$
of $F_{x,z}$ has dimension $f_{k-1}$. Hence any component of $E_{k,x}$
has dimension $f_k-f_{k-1}=\psi_k$  (see Proposition \ref
{defects}). \end{proof}

A different proof of Proposition \ref  {lem:dimel} follows by Proposition
2.2 of \cite {IR}, which asserts that $\Psi_k$ can be seen as a degeneration
of $E_k$.

\begin{remark}\label {rem:elocus} Terracini's Lemma implies that
for $x\in S^ k(X)$ general and
for a general point $z\in E_{k,x}$, one has $T_{X,z}\subseteq T_{S^ k(X),x}$.
Hence $E_{k,x}$ is contained in the tangential $k$--contact locus
$\Gamma_{p_0,\ldots,p_k}$, for all $p_0,\ldots,p_k\in E_{k,x}$
which are smooth points of $X$ and
such that $x\in \langle p_0,\ldots,p_k\rangle$. Again we deduce
$\psi_k\leq \gamma_k$.
\end{remark}

We finish by recalling the following well known \emph{subadditivity theorem}
by Palatini--Zak, whose proof is an application of the previous results (see \cite {Zak}, Chapter V, Proposition 1.7 and Theorem 1.8).  One more piece of notation before that.  Let $X\subset \p^ r$ be an irreducible, non--degenerate variety. One sets

$$k_0:=k_0(X)=\min\{k\in \N: S^ k(X)=\p^ r\}.$$

\begin{theorem}\label {thm:subad}  Let $X\subset \p^ r$ be a smooth, irreducible,
non--degenerate variety of dimension $n$.  One has

\begin{itemize}

\item [(i))] $\psi_1\leq \psi_2\leq \ldots \leq \psi_{k_0}\leq n$;

\item [(ii)] $\psi_k\geq \psi_{k-1}+\psi_1$ for all $k\leq k_0$.

\end{itemize}

\end{theorem}

From Theorem \ref {thm:subad} one deduces

\begin{equation}\label {eq:scorza}
\psi_k\geq k\psi_1, \quad \forall\quad k\leq k_0.\end{equation}

\begin{definition}\label {def:sc} [See \cite {Zak} (Chapter VI, Proposition 1.2)]
A smooth, irreducible,
non--degenerate variety $X\subset \p^ r$ of dimension $n$ with $\psi_1>0$ is called a
 \emph{Scorza variety} if equality
holds in \eqref {eq:scorza} and in addition $k_0=[\frac n \psi_1]$.\end{definition}

The classification
of Scorza varieties is contained in \cite {Zak}, Chapter VI.

\begin{remark}\label {rem:entry1} If $X$ is smooth, and the general entry locus $E_1$
is a quadric, then $E_1$ is smooth (see \cite {FujRob}, pp. 964--65).
To the best  of our knowledge, there is no argument for the smoothness
of the general entry locus $E_k$,  though in \cite {Zak}, p. 123, this is asserted
to be the consequence of "usual general position arguments" which we are unable to
understand. The classification of Scorza varieties in \cite {Zak} seems to depend on this
assertion, which is however false, in its full generality, as the following example shows. \end{remark}

\begin{example}\label{ex:ggg} Consider the scroll $X=S(1,h)\subset \p^ {h+2}$, with $h\geq 4$.
Let $L$ be the line directrix of $X$.
For all $k\geq 1$,
$S^ k(X)$ is the cone with vertex $L$ over $S^ k(Y)$, with $Y$ a rational
normal curve in a $h$--space $\Pi'$ which is skew with $L$. Therefore

$$s^ {(k)}(X)=2k+3<h+2$$
as soon as $k<\frac {h-1}2$ and $X$ is $k$ defective if $k\geq 2$. In this case it is not difficult to see that $E_k$ is formed by
$k+1$ general rulings of $X$ plus the line $L$.
\end{example}

\section{ Zak's theorem on tangencies}

Zak's theorem on tangencies is a basic
tool for the study of projective varieties and their
secant varieties (see \cite  {Zak1}, \cite {Zak}).
It says that
if a linear space $L$ in $\p^r$ is tangent to a {\it smooth}, pure variety
$X\subset\p^r$ along a subvariety $Y$, then

\begin{equation}\label{Zak}
\dim(L)\geq \dim(X)+\dim(Y).
\end{equation}

Zak's original  formulation in \cite {Zak1} is more general, inasmuch as
it applies also to singular varieties $X$: in this case the dimension of the
singular locus of $X$ enters into play.  On the other hand,
as proved in Zak's book \cite {Zak},
formula \eqref{Zak} works also for singular varieties $X$,
provided one takes the \emph{right} definition of
tangency at singular points. Let us recall this definition.

\begin{definition}\label {def:jtan}  Let $X$ be a
variety, $L$ a linear subspace
of $\p^ r$ and $Y$ a subvariety of $X$ contained in $L$.
One says that  $L$ is {\it $J$--tangent}
to $X$ along $Y$ if the following holds. Let $\{(q_0(t),q_1(t))\}_{t\in \Delta}$
be any analytic curve in $X\times Y$ parametrized by the unitary
disc $\Delta$ and such that:

\begin{itemize}
\item [(i)] $q_0(t)\not \in L$ for any $t\in \Delta-\{0\}$;
\item [(ii)] $q_0(0)\in L$.
\end{itemize}
Then the flat  limit of the line $\langle q_0(t), q_1(t)\rangle$ lies in $L$.
\end{definition}

\begin{remark}\label{rem:var} In point (ii) of Definition \ref  {def:jtan}  one may equivalently
ask that  $q_0(0)\in Y$. Indeed, if  $q_0(0)\in L$ but not in $Y$, then it is clear that
the flat  limit of the line $\langle q_0(t), q_1(t)\rangle$ lies in $L$, since
$q_1(0)\in Y$ and $q_0(0)\neq q_1(0)$. 

If $L$ is $J$--tangent to $X$ along $Y$, then it is also
$J$--tangent to $X$ along any subvariety $Z$ contained in $Y$.

If $L$ is $J$--tangent
to $X$ along $Y$, then $L$ is $J$-tangent to any irreducible
component of $X$
along any irreducible component of $Y$.

If  $L$ is  $J$-tangent to $X$ along $Y$,
and $Z$ is a subvariety of $X$ containing $Y$ but not contained in $L$,
then, $L$ is also $J$-tangent to $Z$ along $Y$.
\end{remark}

If $X$ is smooth along $Y$, $J$--tangent is equivalent to
the condition that $L$ contains the tangent space $T_{X,y}$ to $X$ at a
general point $y\in Y$.
If $X$ is singular at some point $y\in Y$, $J$--tangency imposes further
restrictions on $L$.

\begin{example}\label {ex:cone} If $X\subset \p^ r$ is a non--degenerate
cone with vertex $v$, then no proper subspace of $\p^ r$ can
be $J$--tangent to $X$ along a subvariety containing $v$.  \end{example}

The notion of
$J$-tangency provides a suitable setting for
a general formulation of  Zak's
theorem on tangencies valid for singular varieties.

\begin{theorem}\label{ZakSing} [Zak's theorem on tangencies]
Let $X\subset \p^ r$ be a variety. If a linear space $L$ of $\p^r$ is
$J$-tangent to $X$ along a subvariety $Y$, then \eqref {Zak}
holds.
\end{theorem}

It is well known that this theorem is sharp.

\begin{example}  \label{ex:projver1}
There are smooth surfaces $X$ in $\p^ 4$ with a hyperplane
$H$ tangent to $X$ along a curve. An example is the projection
$X$ to $\p^ 4$ of the Veronese surface $V_{2,2}$ in $\p^ 5$ from a
general point $p\in \p^ 5$.

In fact, the Veronese surface has a 2--dimensional system of conics and
there is a hyperplane tangent to $V_{2,2}$ along each conic. Hence, there
is a 1--dimensional system of these tangent hyperplanes passing through
the centre of projection $p$. Such hyperplanes project down to $\p^ 4$
to hyperplanes of $\p^ 4$ which are tangent to $X$ along the corresponding conics.

To be more specific, let $Y$ be any conic on $X$. Consider the pencil of hyperplanes
containing $Y$. This pencil cuts out on $X$, off $Y$ a pencil of conics having a base
point $y$. There is a hyperplane tangent to $X$ along $Y$ if and only if $y\in Y$.

Since the map associating $y$ to $Y$ is
a projective transformation $\omega:(\p^2)^ *\to \p^ 2$,
the locus of points $y\in X$ belonging to the corresponding conic $Y$
describes a conic in $\p^ 2$ and therefore a rational normal quartic on $X$.
\end{example}

Although sharp, Zak's theorem can be improved, as we shall see.
To do this, we first have
to extend the notion of $J$--tangency. We will do this in the next section.

\section{The notion of $J_k$--tangency}\label {sec:jktan}

In order to improve Zak's theorem on tangencies, we need to
extend the concept of $J$--tangency.

\begin{definition} \label {def:jktan} Let $X$, $L$ and $Y$
be as in Definition \ref {def:jtan}.
Let $k$ be a positive integer
such that $\dim(\langle Y\rangle)\geq k-1$.
One says that  $L$ is {\it $J_k$--tangent}
to $X$ along $Y$ if the following holds.
Let
$\{(q_0(t),q_1(t),\ldots,q_s(t))\}_{t\in \Delta}$, with $1\leq s\leq k$,
be any analytic curve in $X\times Y^ s$ parametrized by the unitary
disc $\Delta$ and such that:

\begin{itemize}
\item[(i)] $(q_0(t),\ldots,q_s(t))$ are linearly independent for any $t\in \Delta-\{0\}$;
\item [(ii)] $q_0(t)\not \in L$ for any $t\in \Delta-\{0\}$;
\item [(iii)] $q_0(0)\in L$.
\end{itemize}
Then the flat limit of the $s$--space  $\langle q_0(t), q_1(t),\ldots, q_s(t)\rangle$
for $t\to 0$ lies in $L$.
\end{definition}

\begin{remark} \label {rem:beb} Of course $J_1$--tangency is $J$--tangency.
Remark \ref {rem:var} can be repeated verbatim replacing $J$--tangency with
$J_k$--tangency.

In particular, in point (iii) of Definition \ref {def:jktan} one may equivalently ask
 that $q_0(0)\in J(Y^ s)\cap X$. Indeed, if $q_0(0)$ lies in $L$ but not 
on $J(Y^ s)\cap X$,
then the flat limit of the $s$--space  $\langle q_0(t), q_1(t),\ldots, q_s(t)\rangle$
for $t\to 0$ lies in $L$, since it is spanned by the flat limit $\Pi$ of the $(s-1)$--space  $\langle q_1(t),\ldots, q_s(t)\rangle$, which lies in $L$, and by $q_0(0)$ which lies in $L$ and not on $\Pi$.

Finally $J_k$--tangency implies $J_h$ tangency, for all $h\leq k$, but the converse
does not hold.\end{remark}

\begin{example}\label {ex:projver2} Let us go back to Example \ref {ex:projver1},
from which we keep the notation.
Consider a hyperplane $H$ tangent, and therefore $J$--tangent,
to $X$ along a conic $Y$. Let us show that $H$ is not $J_2$--tangent to $X$
along $Y$. Let $\Pi=\langle Y\rangle$ and let $y\in Y$ be as in
Example \ref {ex:projver1}.  Note that all hyperplanes through $\Pi$
cut on $X$ a curve of the form $Y+Z$, with $Z$ a conic through $y$. This
shows that $\Pi=T_{X,y}$. Take a general conic $Z$ on $X$ through $y$. Set $\Pi'=\langle Z\rangle$.
Then $\Pi\cap \Pi'$ is a line $\ell$ which is the tangent line to $Z$ at $y$
and is also a general line in $\Pi$ through $y$.
Let $x$ be the intersection of $\ell$ with $Y$ off $y$.
Consider an analytic parametrization $p(t)$ of $Z$ around $y$,
so that $p(0)=y\in Y$, and consider the
analytic curve $(p(t),x,y)$ in $X\times Y^ 2$. Then the $2$--space
$\langle p(t),x,y \rangle=\Pi'$ does not depend on $t$,
and therefore its limit is $\Pi'$, which does not lie in
$H$. \end{example}

The notion of $J_k$--tangency will play a crucial role next. Let us add
a couple of related definitions.

\begin{definition}\label {def:smt} Let $X\subset \p^ r$ be variety, $Y$ a
subvariety of  $X$ and let $k$ be a positive integer.
We say that $X$ is \emph{$k$--smooth} along $Y$
if $X$ is smooth along $Y$ and
any subscheme $Z$ of $X$ of finite length $s\leq k+1$
supported at $Y$ spans a linear space of dimension $s-1$.
We say that $X$ is $k$--smooth if it is $k$--smooth
along $X$.
\end{definition}

\begin {definition}\label {def:wr} Let $X\subset \p^ r$ be a variety,
such that $s^ {(k)}(X)<r$.  We will say that $X$ enjoys the \emph{$R_k$--property}
(or briefly that $X$ is a \emph{$R_k$--variety}) if the following holds. For any $i=1,\ldots,k$ and general points $p_0,\ldots,p_i$,
taken in different components of $X$ if $X$ is reducible, the general hyperplane
tangent to $X$ at $p_0,\ldots,p_i$ is $J_i$--tangent to $X$ along $\Gamma_{p_0,\ldots,
p_i}$.\end{definition}

\begin{remark} \label {rem:wz} (a) The notion of $k$--smoothness is \emph{hereditary}, i.e.,
if $X$ is  $k$--smooth along $Y$, then
$X$ is $k$--smooth along any subvariety $Z$ of $Y$.

(b) The notion of $k$--smoothness coincides with $\OOO_X(1)$
being $k$--very ample (see \cite {bs}).
It can also be rephrased as follows: $X$ is $k$--smooth along $Y$
if there is no linear space $L$ of dimension $s<k$ containing
a subscheme $Z$ of $X$ of finite length $\ell\geq s+2$ supported
at $Y$.

Note that $k$--smoothness is a rather rigid notion. For example
a smooth variety containing a line is not $k$--smooth for any $k\geq 2$. On the other hand,
if $X$ is smooth, its $d$--tuple Veronese embedding, with $d\geq k+1$, is $k$--smooth.
\end{remark}

Next we will show a relationship between
the notions of $k$--smoothness and $J_k$-tangency.

\begin {lemma} \label {lem:prep} Let $X\subset \p^ r$ be a 
variety,  $Y$ a  subvariety of $X$. Assume that 
$X$ is smooth along $Y$ and
let $L$ be a linear space tangent to $X$ along
$Y$.  Let
$\{(q_0(t),q_1(t),\ldots,q_s(t))\}_{t\in \Delta}$
be an analytic curve in $X\times Y^ s$
parametrized by the unitary
disc $\Delta$ such that $q_0(t),q_1(t),\ldots,q_s(t)$ are distinct for
$t\in \Delta-\{0\}$ and $q_0(0)\in L$.
Let $Z_0$ be the flat limit of the reduced $0$--dimensional scheme
corresponding to the $0$--cycle
$Z_t= q_0(t)+q_1(t)+\ldots+q_s(t)$. Then $Z_0$ is
contained in $L$.
\end{lemma}

\begin{proof} Let $Z$ be the limit of
the $0$--dimensional scheme corresponding to the $0$--cycle
$q_1(t)+\ldots+q_s(t)$. Note that $Z$ sits in $Y$ and therefore in $L$.
The degrees of $Z$ and $Z_0$ differ by
1. If $Z$ and $Z_0$ do not share the same support, then the assertion
is clear. If $Z$ and $Z_0$ have the same support, then $q_0(0)\in Y$.
Moreover  the ideal sheaves
of $Z$ and $Z_0$ behave as follows: for any point $p$ of the
common support one has
$\II_{Z_0,p}\subseteq \II_{Z,p}$ and the inequality is strict only at one point $q$,
where $ \II_{Z,q}/ \II_{Z_0,q}=\C$. This corresponds to a single condition
imposed to functions in $\II_{Z,q}$
in order to have functions in $ \II_{Z_0,q}$.
This is clearly a \emph{tangential condition}, i.e. the functions in $ \II_{Z,q}$,
in order to be in $ \II_{Z_0,q}$,
are required to be annihilated by the differential operator corresponding
to the tangent vector to the branch of the curve $\{q_0(t)\}_{t\in \Delta}$ at $t=0$.
Since $L$ is tangent to $X$ along $Y$,
 this condition is verified by the
equations of $L$ as well, proving the assertion. \end{proof}

\begin{proposition}\label {prop:smtoreg}  Let $X\subset\p^r$ be a variety, 
$Y$ a subvariety. Assume that $X$ is $k$--smooth along $Y$. Then a linear space
$L$ is $J_k$--tangent to $X$ along $Y$ if and only if
it is tangent to $X$ along $Y$.
\end{proposition}

\begin{proof} Assume $L$ is tangent to $X$ along $Y$.
Let  $\{(q_0(t),q_1(t),\ldots,q_s(t))\}_{t\in \Delta}$, with $1\leq s\leq k$,
be any analytic curve in $X\times Y^s$ as in Definition \ref {def:jktan}.

Suppose the limit $\Pi_0$ of the $s$--space
$\Pi_t=\langle q_0(t), q_1(t),\ldots, q_s(t)\rangle$ does not lie in $L$.
Then $\Pi_0\cap L$ would be a linear space of dimension
$t<s$ containing the  scheme $Z_0$ (see Lemma \ref  {lem:prep}).
This contradicts the
$k$--smoothness assumption.
\end{proof}

\begin{remark}
The converse of Proposition \ref {prop:smtoreg} does not hold,
i.e.  there are varieties which
are $k$-regular but not $k$-smooth.
For example, the Segre variety  ${\rm Seg}(n,n)$,  with $n\geq 3$, is not $2$-smooth
because it contains lines, hence also triples of collinear points,  but it  is a $R_{n-1}$-variety
(see the Example \ref {ex:veretc}  below).
\end{remark}

Next we point out a couple of easy lemmata.

\begin{lemma}\label{project} Let $X, Y, k$ be
as in Definition \ref {def:smt}.
Let $L$ be a linear space which is $J_k$--tangent to $X$ along $Y$.
Fix a point $p\notin S^k(X)$ and let $\pi$ be the projection from $p$.
Then $L'=\pi(L)$ is $J_k$--tangent to $X'=\pi(X)$ along $Y'=\pi(Y)$.
\end{lemma}

\begin{proof} Let
$\{(q_0(t),q_1(t),\ldots,q_s(t))\}_{t\in \Delta}$, with $1\leq s\leq k$,
be a analytic curve as in Definition \ref {def:jktan}. Then
 the limit $\Pi$ of the $s$--space  $\langle q_0(t), q_1(t),\ldots, q_s(t)\rangle$
for $t\to 0$ lies in $L$ and does not contain $p$.

Consider the curve $\{(\pi(q_0(t)),\pi(q_1(t)),\ldots,\pi(q_s(t)))\}_{t\in \Delta}$
in $X'\times Y'^ s$. It enjoys properties (i)--(iii) of Definition \ref {def:jktan}
with $Y, L$ replaced by $Y', L'$.
The limit of the $s$--space $\langle \pi(q_0(t)),\pi( q_1(t)),\ldots, \pi(q_s(t))\rangle$
is the projection of $\Pi$ from $p$, hence it is an $s$--space contained in
$L'$. On the other hand, any curve in $X'\times Y'^ s$
enjoying properties (i)--(iii) of Definition \ref {def:jktan}
with $Y, L$ replaced by $Y', L'$ can be obtained in this way. \end{proof}

\begin{lemma}\label{project2} Let $X, Y, k$ be
as in Definition \ref {def:smt} with $k\geq 2$.
Let $L$ be a linear space which is $J_k$--tangent to $X$ along $Y$.
Let $p$ be a point on $Y$ and let $\pi$  be the projection from $p$.
Then $L'=\overline {\pi(L)}$ is $J_{k-1}$--tangent to
$X'=\overline {\pi(X)}$ along $Y'=\overline {\pi(Y)}$.
\end{lemma}

\begin{proof} Suppose the assertion is not true. Then we can find
an analytic curve
$\{(q_0(t),\ldots,q_s(t))\}_{t\in \Delta}$, with $1\leq s\leq k-1$,
in $X'\times Y'^ s$ verifying (i)--(iii) of Definition
\ref {def:jktan} with $Y, L$ replaced by $Y', L'$,
and such that the limit $\Pi'$ of the $s$--space
$\langle q_0(t), q_1(t),\ldots, q_s(t)\rangle$
for $t\to 0$ does not lie in $L'$.
By slightly perturbing this curve if necessary,
we may assume that it
can be lifted to an analytic curve $\{(p_0(t),\ldots,p_s(t))\}_{t\in \Delta}$ in
$X\times Y^ s$. Consider then the analytic curve
$\{(p_0(t),\ldots,p_s(t),p_{s+1}(t))\}_{t\in \Delta}$ in
$X\times Y^{s+1}$, with $p_{s+1}(t)$ independent on $t$ and
equal to $p$. This curve verifies (i)--(iii) of Definition \ref {def:jktan}
and therefore
 the limit $\Pi$ of the $(s+1)$--space  $\langle p_0(t),\ldots, p_s(t),p\rangle$
for $t\to 0$ lies in $L$. But then its projection from $p$, which is $\Pi'$,
should lie in $L'$, a contradiction.
 \end{proof}

 \begin{remark} As in Lemma \ref {project}, one has the following.
Let $X$ be a $R_k$--variety [resp. a $k$-smooth variety] and fix a point
$p\notin S^k(X)$. Then the image of the projection $\pi$
of $X$ from $p$ is again a $R_k$-variety  [resp. a $k$--smooth variety].

Similar considerations hold for Lemma \ref {project2}.
\end{remark}

The following provides a simple criterion for the $R_k$--property.

\begin{proposition}\label{prop:criter} Let $X\subset \p^ r$ be an irreducible,
non--degenerate, variety, such that $s^ {(k)}(X)<r$. Assume that $\gamma_i=
\psi_i$ for all $i=1,\ldots, k$. Assume moreover that the intersection of
the indeterminacy loci of all tangential projections $\tau_{p_0,\ldots,p_i}$
is empty, for $p_0,\ldots,p_i$ general points in $X$.
Then $X$ is a $R_k$--variety.
\end{proposition}

\begin{proof}  Fix any $i=1,\ldots,k$ and general points $p_0,\ldots,p_i$.
Consider the $i$--tangential projection $\tau_i$ of $X$ from
$T_{X,p_1,\ldots,p_i}$ and set $p=\tau_i(p_0)$. Note that
$\gamma_i=\psi_i$ implies that $\Psi_{p_0,\ldots,p_i}$ is the irreducible
component of $\Gamma_{p_0,\ldots,p_i}$ containing $p_0$.
Assume $\Gamma_i=\Gamma_{p_0,\ldots,p_i}$ is irreducible.
The argument in the reducible case is the same, and can be left to
the reader. Thus $p$ is the image via $\tau_i$ of  
$\Gamma_i=\Gamma_{p_0,\ldots,p_i}$.
Consider a general hyperplane $H$
 tangent to $X$ along $\Gamma_i$. Let $H'$ be the image of $H$ via $\tau_i$,
 tangent to $X_i$ at $p$, which is a smooth point of $X_i$.

Take a curve  $\{(q_0(t),q_1(t),\ldots,q_s(t))\}_{t\in \Delta}$, with $1\leq s\leq i$,
in $X\times \Gamma_i^ s$ verifying (i)--(iii) of Definition
\ref {def:jktan}  with $Y=\Gamma_i$. Choose $p_0,\ldots,p_i$ on
$\Gamma_i$ in such a way that none of the points $q_0(t),q_1(t),\ldots,q_s(t)$ for $t$
general in $\Delta$, sits in the indeterminacy locus of $\tau_i$.
Then the projection via
$\tau_i$ of the limit $\Pi$ of the $s$--space  $\langle q_0(t), q_1(t),\ldots, q_s(t)\rangle$
for $t\to 0$ sits in $T_{X_i,p}$, which in turn sits in $H'$.  This
implies that $\Pi$ sits in $H$, proving the assertion.
\end{proof}

\begin{example}\label{ex:veretc} The previous proposition implies that the following
varieties are $R_k$--varieties:

 \begin {itemize}

\item [(i)] the $(k+1)$--dimensional Veronese variety $V_{2,k+1}$ in $\p^ {\frac {k(k+3)}2}$;

\item [(ii)] the $2(k+1)$--dimensional Segre variety ${\rm Seg}(k+1,k+1)$ in $\p^ {k^ 2+4k+3}$;

\item [(iii)] the $4(k+1)$--dimensional Grassmann variety $\G(1,2k+3)$ in $\p^ { {2k+4}\choose 2-1}$.

 \end{itemize}

Indeed, the $i$--contact locus of $V_{2,k+1}$ is the Veronese image of a general linear subspace of
dimension $i$. This is also the fibre of the general $i$--tangential projection of $V_{2,k+1}$.

Similarly the $i$--contact locus of ${\rm Seg}(k+1,k+1)$ is a subvariety of type
${\rm Seg}(i,i)$, which is also the fibre of  the general $i$--tangential projection of ${\rm Seg}(k+1,k+1)$.

Finally the $i$--contact locus of $\G(1,2k+3)$  is a subvariety of type
$\G(1,2i+1)$, which is also the fibre of  the general $i$--tangential projection of ${\rm Seg}(k+1,k+1)$.

The condition about the indeterminacy loci is easy to be verified in all these cases.
\end{example}

\section{An extension of Zak's theorem on tangencies}\label {sec:exttan}
The notion of $J_k$--tangency plays a basic role in the
following extension of Zak's theorem on tangencies.

\begin{theorem}\label{Zakk}
Let $X\subset \p^ r$ be a non--degenerate  variety
and let $L\neq\p^r$ be a proper linear subspace of $\p^ r$ which
is $J_k$--tangent to $X$ along a pure subvariety $Y$.
Then:
\begin{equation}\label{zakk}
\dim(L)\geq \dim(X)+\dim(J^{k}(Y))\geq  \dim(X)+\dim(S^{k-1}(Y)).
\end{equation}
In particular, \eqref {zakk} holds when $L$ is
tangent to $X$ along $Y$ and $X$ is $k$--smooth along $Y$.
\end{theorem}

\begin{proof}
The proof follows Zak's original argument.

By definition we have $\dim(\langle Y\rangle)\geq k-1$.
Fix $k$ independent points $p_1,\ldots,p_k\in Y$ such that
$N:=\langle p_1,\dots,p_k\rangle$ lies in a component $Z\subseteq J^{k}(Y)$
of maximal dimension. Fix $p_0\in X$ general,
so that $p_0\notin L$. Set $M:=\langle p_0,p_1,\dots,p_k\rangle$, so that
$\dim(M)=k$.

We may assume that $M$ is not contained in $X$, otherwise $X$
would be a cone with vertex  $N$,
contradicting the fact that $L$ is  $J_k$--tangent to $X$ along $Y$
(see Example \ref {ex:cone}).

Pick a general point $x\in M-X$ and let
$f: X\times J^{k}(Y)\to L\times L$ be the morphism which is
the inclusion on the second coordinate and
the projection $\pi$ from a general linear space $\Pi$ of dimension
$r-\dim(L)-1$ containing $x$
on the first coordinate. By the generality assumption
one has  $\Pi\cap L=\Pi\cap X=\emptyset$.

We claim that  $f$ is finite.  Suppose in fact $C$ is a curve
mapping to a point via $f$. Its projection on $X$ would be a curve $C'$,
since $f$ is injective on the second coordinate. Moreover $\pi(C')=y$
would be a point, and therefore $C'\subset \langle y,\Pi\rangle$.
But then $X\cap \Pi\supseteq C'\cap \Pi\neq \emptyset$, a contradiction.

Now consider the component $X_0$ of $X$ passing through $p_0$
and restrict the map $f$ to $X_0\times Z$.
If $\dim(L)< \dim(X)+\dim(J^{k}(Y))=\dim(X_0)+\dim(Z)$,
Fulton--Hansen's connectedness theorem  (see \cite{fh}) implies that the inverse image
of the diagonal $D$ of $L\times L $ is connected.

Notice that $\pi(p_0)=y\in N$. Thus $f(p_0,y)=(y,y)$.
Since also $(x,x)$, with $x\in X\cap J^{k}(Y)$,
belongs to the inverse image of
$D$, Fulton--Hansen's theorem implies that
there is a curve $\{(q_0(t),q_1(t),\ldots,q_k(t))\}_{t\in
\Delta}$ in $X_0\times Y^ s$ such that:

\begin{itemize}
\item [(i)] $q_1(t),\ldots,q_s(t)$ are linearly independent for any $t\in \Delta-\{0\}$;
\item [(ii)] $q_0(t)$ is a general point in $X_0$ for $t\in \Delta-\{0\}$.
\item [(iii)] $\pi(q_0(t))\in Q_t:=\langle q_1(t),\ldots,q_k(t))\rangle\subset Z$ for all $t\in \Delta-\{0\}$;
\item [(iv)] $q_0(0)\in J^ k(Y)\cap X$.
\end{itemize}

The $J_k$--tangency hypothesis implies that
the limit $P$ of  $P_t:=\langle q_0(t), q_1(t),\ldots, q_s(t)\rangle$
for $t\to 0$ lies in $L$.
On the other hand, since  $\pi(q_0(t))\in Q_t$ for all $t\in \Delta-\{0\}$, then
$\Pi\cap P_t\neq \emptyset$
for all $t\in \Delta-\{0\}$, and therefore $\Pi\cap P\neq \emptyset$.
Thus $L\cap\Pi\neq\emptyset$, a contradiction.
\end{proof}

\begin{remark} Theorem \ref {Zakk} is sharp. The Veronese
surface $V_{2,2}$ has hyperplanes $L$ which are
$J_2$--tangent along a conic $Y$, because $V_{2,2}$ is 2--smooth. 
In that case $S^ 1(Y)$ is a plane and in
\eqref {zakk} equality holds. This extends to higher
Veronese varieties $V_{2,r}$, $r\geq 3$.

The $J_k$--tangency hypothesis is essential in Theorem \ref {Zakk}.
Indeed, for a general
projection $X$ of $V_{2,2}$ in $\p^ 4$, the inequality \eqref {zakk} does not hold, but
$X$ is not 2--smooth (since it has trisecant lines, see Example \ref {ex:projver1}).
\end{remark}

\section{An extension of Zak's theorem  on linear normality}\label {sec:ln}

A striking consequence of the theorem on tangencies is the famous:

\begin{theorem}\label{thmln} [Zak's theorem on linear normality]
Let $X\subset \p^ r$ be a smooth, irreducible,
non--degenerate variety of dimension $n$. Then:

$$s(X)\geq \min\{r,\frac 32 n+1\}.$$

\end{theorem}

The reason for the name of the theorem, is that it gives a positive answer to the
following conjecture by Hartshorne (see \cite {???}):

\begin{conjecture}\label {conj:hartb}  [Hartshorne's conjecture on linear normality]
Let $X\subset \p^ r$ be a smooth, irreducible,
non--degenerate variety of dimension $n$. If $3n>2(r-1)$ then $X$ is linearly
normal.
\end{conjecture}

In this section we want to extend Zak's Theorem \ref {thmln}, by giving a lower bound
on $s^ {(k)}(X)$, under suitable
assumptions for the variety $X\subset \p^ r$.

Let us prove the following key lemma:

\begin{lemma}\label{step2}  Let $X\subset \p^ r$ be a non--degenerate variety,
such that $s^ {(k)}(X)<r$. Assume $X$ is a $R_k$--variety
(or $X$ is $k$--smooth). Then:

\begin{equation} \label{step2eq}
2f_k(X)\leq kn. 
\end{equation}
Moreover, if the equality holds, then:
\begin{equation}
\begin{matrix}\label{claim1}
(i)  &  \gamma_i=\psi_i=if, \mbox{ for all } i=1,\ldots,k;\\
(ii) &  f_i=\frac  {i(i+1) }2 f, \mbox{ for all } i=1,\ldots,k;\\
(iii) & n=(k+1)f. \phantom{\mbox{ for all } i=1,\ldots,k}
\end{matrix}
\end{equation}
where, as usual, we set $f=f_1$.
\end{lemma}

\begin{proof} After projecting generically and applying Lemma \ref {project}, we may reduce
ourselves to the case $r=nk+n+k-f_k+1$.

Consider the general tangential $i$--contact locus $\Gamma_i=\Gamma_{p_0,\ldots,p_i}$,
for all $i=1,\ldots,k$. By
the hypothesis, the general hyperplane $H_i$  tangent to $X$ at $p_0,\ldots,p_{i-1}$ is $J_{i-1}$--tangent
to $X$ along $\Gamma_{i-1}$. Moreover it does not contain $\Gamma_i$. Hence it is also
$J_{i-1}$ tangent to $\Gamma_i$ along $\Gamma_{i-1}$.
Since  $S^ {(i-1)}(\Gamma_i)$
does not fill up $\Pi_i$ (see Proposition \ref {prop:tandef}), after projecting from
a general point of $H_i\cap\Pi_i$, by Lemma \ref {project} we find a linear space of dimension $\dim(H_i\cap \Pi_i)-1=\dim(\Pi_i)-2$, which is $J_{i-1}$-tangent to the projection of $\Gamma_i$ along $\Gamma_{i-1}$. Then, by  Proposition \ref {nofill} and  Theorem \ref {Zakk}, we deduce that:
$$
\dim(\Pi_i)-2\geq \dim(\Gamma_i)+\dim(S^{i-2}(\Gamma_{i-1})).
$$
Using  Proposition \ref {prop:tandef}, this inequality translates into: 
\begin{equation}\label {eq:xx}
 i\gamma_i\geq (i-1)\gamma_{i-1}+f_i-f_{i-2}
 \end{equation}
for all $i=2,\ldots,k$, where we put $f_0=0$. Adding these relations up, and
taking into account that $\gamma_1\geq f_1=\psi_1$, we deduce:
\begin{equation} \label {eq:cc} k\gamma_k\geq f_k+f_{k-1}.\end{equation}
Now notice that there exists  a hyperplane $H$ in $\p^r$ which is $J_k$--tangent to
$X$ along $\Gamma_k$. With the same argument as above, we find:
$$ nk\geq k\gamma_k +f_k-f_{k-1}.$$
By taking into account \eqref  {eq:cc}, \eqref {step2eq} follows.

Suppose $2f_k=kn$. Then equality holds in \eqref{eq:xx} for all $i=1,\dots,k$.
In particular we get $\gamma_1=\psi_1=f$. Thus (i) of \ref 
{claim1} holds for $i=1$. Assume $i\geq 2$ and proceed by induction.

By Proposition \ref{defects}, we know that $\gamma_i\geq\psi_i=f_i-f_{i-1}$. This, together with \eqref {eq:xx} (where we must have an equality), yields
$$ (i-1)\gamma_i\leq (i-1)\gamma_{i-1}+f_{i-1}-f_{i-2}$$
so that, by induction,
$$ (i-1)\gamma_i\leq (i-1)\gamma_{i-1}+\psi_{i-1} =i\gamma_{i-1}=i(i-1)f,$$
thus $if\geq \gamma_i$.
On the other hand, by induction and subaddivity (see Theorem \ref{thm:subad}), we have
$$\gamma_i\geq \psi_i\geq \psi_{i-1}+\psi_1= \gamma_{i-1}+f=if$$
so that $if=\gamma_i$.
This proves (i) and (ii) immediately follows; moreover $kn=2f_k=k(k-1)f$ and also (iii) is proved.
\end{proof}

\begin{example} There are examples of $R_k$-varieties $X$ for which formulas 
\eqref{claim1}, (i), (ii) hold, but $2f_k(X)<kn$ and $n\neq (k+1)f$.

For instance, take $X$ to be the Segre variety ${\rm Seg}(3,4)$ in $\p^{19}$.
The variety $S^2(X)$ has dimension $17$. The first tangential projection sends
$X$ to ${\rm Seg}(2,3)\subset\p^{11}$. The second tangential projection sends
$X$ to ${\rm Seg}(1,2)\subset\p^{5}$. One computes $f=f_1=\gamma_1=\psi_1=2$, 
$\gamma_2=\psi_2=4=2f$, $f_2=6=3f$. Moreover, using Proposition \ref{prop:criter}
one sees that $X$ is an $R_2$-variety.

On the other hand, $n=7\neq 6= (k+1)f$, and $2f_k=12<14=nk$.
\end{example}

We can now prove our extension of Zak's linear normality theorem.

\begin{theorem}\label {thm:extln} Let $X\subset \p^ r$ be a
non--degenerate variety of dimension $n$.
Assume $X$ is a $R_k$--variety (or $X$ is $k$--smooth). Then

$$\mbox{ either } \quad S^k(X)=\p^r \quad\mbox{ or }
\quad s^{(k)}(X)\geq \frac {k+2}2 n + k.$$
\end{theorem}

\begin{proof} Let $s^{(k)}(X)<r$.
By \eqref {step2eq}, one has

\begin{multline}\label{formuladim}
\qquad s^{(k)}(X)=(k+1)n+k-f_k\geq \\ \qquad\geq\
(k+1)n+k-\frac{kn}2 \  =\  \frac {k+2}2 n + k.\qquad
\end{multline}
\end{proof}

\begin {corollary}\label {cor:lnext}
Let $X\subset \p^ r$ be a
$k$--smooth variety of dimension $n$. If
$(k+2)n>2(r-k)$ then $X$ is linearly normal.
\end{corollary}

\begin{proof} If $X$ is not linearly normal,
then it comes as an isomorphic projection of a variety $X'\subset \p^{r+1}$
from a point $p\not\in X'$.  Hence $X'$ is also $k$--smooth and therefore
a $R_k$--variety. Then Theorem  \ref {thm:extln}
implies that  $S^ k(X')=\p^ {r+1}$. Therefore there is some
$(k+1)$--secant $k$--space to $X'$ passing through the centre of projection,
yielding a $(k+1)$--secant $(k-1)$--space to $X$, contradicting $k$--smoothness.
\end{proof}

We finish this section by stressing that the $R_k$--property
in Theorem  \ref {thm:extln} is really essential, as the following example due
to C. Fontanari shows.

\begin{example}\label {ex:fon} Consider the rational normal scroll
3--fold $X:= S(1,1,h)\subset \p^ {h+4}$, $h\geq 2$. Note that we have
two line sections on $X$ spanning a $3$--space $\Pi$. For all $k\geq 1$,
$S^ k(X)$ is the cone with vertex $\Pi$ over $S^ k(Y)$, with $Y$ a rational
normal curve in a $h$--space $\Pi'$ which is skew with $\Pi$. Therefore

$$s^ {(k)}(X)=2k+5<h+4$$
as soon as $k<\frac {h-1}2$. On the other hand $2k+5$ is smaller than
the bound $\frac {k+2}2  n + k=\frac 52k+3$ as soon as $k>4$.

Indeed, $X$ is not a $R_k$--variety. For instance, consider the
case $h=2k+2$. Then there is a unique  hyperplane $H$ which is
tangent to $X$  at  $k+1$ general points $p_0,\ldots,p_k$: it contains
$\Pi$ and projects to the unique hyperplane $H'$ in $\Pi'$ which is
tangent to $Y$  at  the projection points $p'_0,\ldots,p'_k$ of
$p_0,\ldots,p_k$ from $\Pi$. The $k$--contact locus $\Sigma$ contains
the union the rulings of the scroll $X$ passing through $p_0,\ldots,p_k$.
Suppose $H$ is $J_k$-tangent to $X$ along $\Sigma$. By projecting down
from three general points of $\Pi$ we would have $J_{k-3}$--tangency of
the image hyperplane to a cone over $Y$, a contradiction (see Lemma \ref {project2}
and Example \ref {ex:cone}). \end{example}

\section{The classification theorem}\label{sec:class}

Let us start with the following definition.

\begin{definition}\label {def:ksev} A \emph {$k$--Severi variety} is an irreducible,
non--degenerate $R_k$--variety $X\subset \p^ r$, such that

\begin{equation}\label{eq:boundsev}
r>s^ {(k)}(X)=\frac {k+2}2 n + k.\end{equation}

\end{definition}

A $1$--Severi variety is simply called a \emph{Severi variety}.
Note we do not require smoothness of $X$ in Definition \ref  {def:ksev}.
We will see in a moment that $k$--Severi varieties are smooth, thus,
in case $k=1$,
our definition of Severi varieties turns out to coincide
with the one of Zak (see  \cite {Zak}).

\begin{remark}\label{rem:st}  A trivial, but useful, remark is that $k$--Severi varieties
are not cones, because of the $R_k$--property.
\end{remark}

Another striking result of
 Zak's is the famous classification theorem.

 \begin {theorem}\label{thm:zakclassthm} [Zak's Classification Theorem]
 Let $X\subseteq \p^ r$ be a smooth Severi variety. Then $X$ is one of the following varieties:

 \begin {itemize}

\item [(i)] the Veronese surface $V_{2,2}$ in $\p^ 5$;

\item [(ii)] the $4$--dimensional Segre variety ${\rm Seg}(2,2)$ in $\p^ 8$

\item [(iii)] the $8$--dimensional Grassmann variety $\G(1,5)$ in $\p^ {14}$

\item [(iv)] the $16$--dimensional $E_6$--variety in $\p^ {26}$.

 \end{itemize}

 \end{theorem}

 \begin{remark}\label {rem:sevsev} Case (i)
 of Theorem \ref {thm:zakclassthm} is due
 to Severi (see \cite {sev}), whence the denomination of Severi varieties.

 Recall that Severi varieties are related to the unitary composition algebras
 $\R, \C, \Ha, \Oa$. If $\AAA$ is one of these algebras, then take all $3\times 3$
 hermitian matrices  $A$ and impose that $\rk(A)=1$. This gives  equations
 defining the Severi varieties. The secant variety to a Severi variety is defined by
 the vanishing of $\det (A)$. Note that $\Oa$ being non--associative, the existence of
 this \emph{determinant} is somewhat exceptional. Indeed there is no analogue for
 higher order matrices.
 \end{remark}

We devote this section to the analogous classification
of $k$--Severi variety for $k\geq 2$.

\begin{lemma}\label {cor:all} Let $X\subset \p^ r$ be a $k$--Severi variety. Let $p_0,\ldots,p_k\in X$ be general points and $x\in \langle p_0,\ldots,p_k\rangle$ be a general point of $S^k(X)$. Then $\Psi_{p_0,\ldots,p_k}$ and  $E_{k,x}$ are irreducible components of  $\Gamma_{p_0,\ldots,p_k}$.
\end{lemma}
\begin{proof} It follows from \eqref {claim1}, (i), of Lemma \ref {step2eq}.
\end{proof}

\begin{lemma}\label{lem:inter} Let $X\subset \p^ r$ be a $k$--Severi variety. Let $p_0,\ldots,p_k\in X$ be general points and set, as usual, $\Gamma_i=\Gamma_{p_0,\ldots,p_i}$ and $\Pi_i=\langle \Gamma_i\rangle$. Then for all $i=1,\ldots,k$ one has:

\begin{equation}\label {eq:apt}
T_{\Gamma_i,p_1,\ldots,p_i}=T_{X,p_1,\ldots,p_i}\cap \Pi_i.
\end{equation}

\noindent Moreover  the intersection of $\Pi_i$ with $X$ coincides with $\Gamma_i$.\end{lemma}
\begin{proof} One has

$$T_{\Gamma_i,p_1,\ldots,p_i}\subseteq T_{X,p_1,\ldots,p_i}\cap \Pi_i.$$

By formulas \eqref{claim1} and Proposition \ref {prop:tandef}, 
$T_{\Gamma_i,p_1,\ldots,p_i}$ has codimension 1 in $\Pi_i$. Hence, if \eqref  {eq:apt} would not hold, then $T_{X,p_1,\ldots,p_i}$ would contain $\Pi_i$, and therefore $\Gamma_i$. Thus it would contain $p_0$, i.e. general point of $X$,  a contradiction.
The final assertion follows  from the fact that $\Gamma_i$ is the general fibre of the general tangential projection $\tau_{i-1}$.\end{proof}

\begin{theorem}\label {thm:ksev} Let $X\subset \p^ r$ be a $k$--Severi variety of dimension
$n$. Then $X$ is smooth.
\end{theorem}

\begin{proof} After a general projection, we may assume that $r=\frac {k+2}2 n + k+1$.

Let $p_1,\ldots,p_k\in X$ be general points. Let $\tau_k$ be the $k$--tangential projection from $T_{X,p_1,\ldots,p_k}$. Its image $X_k$ has dimension  $n-\psi_k$ and spans  a projective space of dimension $r-1-s^{(k-1)}= r-1-(kn+k-1-f_{k-1})$. Using \eqref {claim1}
one sees that $X_k$ is a non--linear hypersurface in $\p^ {f+1}$.

Let $q_0,q_1\in X$ be general points. By \eqref{claim1}, 
and since the general secant line to
$X$ is not a trisecant  (see \cite {Zak}),  the
tangential 1--contact locus $\Gamma_1=\Gamma_{q_0,q_1}$ is a $f$--dimensional quadric.
Indeed, by Proposition \ref {prop:tandef} we have $\dim(\Pi_1)=f+1$.

\begin {claim}\label {claim:onebis} The tangential projection $\tau_k$ isomorphically maps
$\Gamma_1$ to $X_k$.  Then $X_k$ and  $\Gamma_1$
are smooth quadrics. Moreover the general  tangential 1-contact locus intersects the general
tangential $k$--contact locus $\Gamma_k=\Gamma_{p_0,\ldots,p_k}$ transversally at one point.
\end{claim}

\begin{proof}[Proof of the Claim] In order to prove the first assertion, it suffices to show that $\Pi_1=\langle \Gamma_1\rangle$ does not intersect $T_{X,p_1,\ldots,p_k}$. We argue by contradiction, and assume that $T_{X,p_1,\ldots,p_k}\cap \Pi_1\neq \emptyset$. If this happens, then:

\begin{itemize}
\item [(i)] either $\tau_k(\Gamma_1)$ is a subspace of dimension at most $f$,
\item [(ii)] or $\Gamma_1$ is  singular,  $T_{X,p_1,\ldots,p_k}\cap \Pi_1$ is a subspace of the vertex of $\Gamma_1$ and $\tau_k(\Gamma_1)$ is a quadric
of dimension smaller than $f$.
\end{itemize}

Case (i) is impossible. Indeed given two general
points of $X_k$ there would be a subspace $\tau_k(\Gamma_1)$ containing them and sitting inside $X_k$, contradicting the non--degeneracy of  $X_k$  in $\p^ {f+1}$.
In case (ii), the general
tangential $k$--contact locus $\Gamma_k$, i.e. the general fibre of $\tau_k$,
intersects the singular quadric $\Gamma_1$ in a positive dimensional subspace strictly
containing $T_{X,p_1,\ldots,p_k}\cap \Pi_1$ and not contained in the vertex of $\Gamma_1$.
Consider then the projection $\pi:X\map \p^ n$ from $\Pi_k=\langle \Gamma_k\rangle$ and let $X'$ be its image, which is non--degenerate in $\p^ n$.

By the above considerations, the image of $\Gamma_1$
via $\pi$ would be a linear subspace of  dimension $f'<f$. Therefore $\dim(X')<n$ and
moreover two general points of $X'$ would be contained in a linear subspace of dimension $f'$ contained in $X'$. This contradicts the non--degeneracy of $X'$.

As for the second assertion, note that $X_k$ is a quadric, which is smooth, otherwise we have a contradiction to Lemma \ref{cor:all}. The final assertion also follows from Lemma \ref{cor:all}
since the above argument implies that $\Gamma_1$ intersects the general fibre of $\tau_k$, i.e. $\Gamma_k$,  transversally in one point. \end{proof}

As a consequence we have:

\begin{claim}\label {claim:apt} For all $i=2,\ldots,k$, the general tangential $i$--contact locus $\Gamma_i$ is a $(i-1)$--Severi variety.\end{claim}

\begin{proof} [Proof of the Claim] The irreducibility of the general $1$--contact loci, implies the irreducibility of the higher contact loci $\Gamma_i$, with $i\geq 2$. Moreover, by Lemma \ref {lem:inter} and by the $R_k$--property, one has that $\Gamma_i$ is a $R_{i-1}$--variety and Lemma \ref{step2} and Proposition \ref {prop:tandef}  yield  that $\Gamma_i$ is a $(i-1)$--Severi variety.
\end{proof}

Next we have:

\begin{claim}\label{claim:tap} For all $i=2,\ldots,k$, and for the general tangential $i$--contact locus $\Gamma_i$, the secant variety $S^ {i-1}(\Gamma_i)$ is not a cone. Similarly $S^ k(X)$ is not a cone.
\end{claim}

\begin{proof} [Proof of the Claim]   We prove the assertion for $\Gamma_i$, the proof for $X$ is similar.
Suppose $S^ {i-1}(\Gamma_i)$ is a cone with vertex $p$.
Since $\Gamma_i$ is not a cone (see Remark \ref {rem:st}), there is a maximum positive integer $j<i$ such that $p$ is not a vertex of $S^ {j-1}(\Gamma_i)$. Then $p$ does not sit in the indeterminacy locus of the general projection $\tau_{j}=\tau_{X,p_1,\ldots,p_{j}}$. The image $Z$  of
 $\Gamma_i$ via $\tau_{j}$ has dimension $(i-j)f$ and is a cone with vertex at
 the image of $p$. Hence the general tangent hyperplane to $Z$ is tangent along a positive
 dimensional variety. This implies that the general tangential $j$--contact locus of
 $\Gamma_i$, hence of $X$, has dimension at least $jf+1$, which is a contradiction.\end{proof}

\begin{remark} The previous argument provides a sort of converse to the criterion
of Proposition \ref{prop:criter}. 

Indeed, it  proves that if a $R_k$-variety satisfies $\gamma_i=\psi_i=if$
for all $i=1,\dots,k$, then the intersection of the indeterminacy loci
of all tangential projections $\tau_{p_0,\dots,p_i}$ is empty, for $i=1,\dots,k$.

\end{remark}

 Now we improve Lemma \ref {lem:inter}.

 \begin{claim}\label{claim:alf} For all $i=2,\ldots,k$, the general tangential $i$--contact locus $\Gamma_i$ is smooth. Furthermore $X$ is smooth along $\Gamma_k$ and
  $\Gamma_k$ is the schematic intersection of $\Pi$ with $X$.
\end{claim}

\begin{proof} [Proof of the Claim]   Suppose $\Gamma_i$  is singular at a point
 $p$. Consider the general tangential projection $\tau_i=\tau_{X,p_1,\ldots,p_i}$. By generic
 smoothness of $\tau_i$ and by Lemma \ref {cor:all} , $p$ has to be in the indeterminacy locus of $\tau_i$, i.e. in the intersection of $T_{X,p_1,\ldots, p_i}$ with $\Gamma_i$.
 By Lemma \ref {lem:inter}, this coincides with the intersection of $T_{\Gamma_i,p_1,\ldots, p_i}$ with $\Gamma_i$. By the genericity of $p_1,\ldots,p_i$ and Terracini's Lemma, we deduce that $S^ {i-1}(\Gamma_i)$ is a cone with vertex $p$, contradicting Claim \ref {claim:tap}.

 Let now $q\in \Gamma_k$ be any point and let $p_1,\dots,p_k\in \Gamma_k$ be general points. Thus
 $p_1,\dots,p_k$ are general points on $X$.  Moreover $\Gamma_k$ is the fibre of $q$ in the tangential projection $\tau_k$. Since the image $X_k$ of $X$ via $\tau_k$ is smooth of dimension $f$ and the fibre of $q$ via $\tau_k$ is smooth of dimension $\psi_k=kf$ at $q$ we see that $X$ is smooth of dimension $n =(k+1)f$ at $q$. The final assertion follows by the same argument.  \end{proof}

Now we go back to the projection $\pi: X\map \p^ n$ from $\Pi_k$.

\begin {claim}\label {claim:four} The map $\pi: X\map \p^ n$ is birational.\end{claim}

\begin{proof}[Proof of the Claim] Let $X'$ be the image of $X$.
By Claim \ref {claim:onebis}, the restriction of $\pi$ to a general
tangential 1--contact locus $\Gamma_1$ is birational,  and its image
is a $f$--subspace of $\p^ n$ contained in $X'$. This yields that $X'$ is a subspace of $\p^ n$, and therefore $X'=\p^ n$.

Assume by contradiction that $\pi$ is not birational. Then, if $x\in X$ is a general point,  there is a point $y\in X$, with $x\neq y$, such that $\pi(x)=\pi(y)$.
Note that there is some $f$--dimensional quadric $\Gamma$ containing $x$ and $y$ and contained in $X$, i.e. a flat limit of $\Gamma_{x,z}$ with $z$ a general point of $X$ tending to $y$. The quadric $\Gamma$, as well as $\Gamma_{x,z}$,  has a not empty intersection with  $\Pi_k$ (see Claim \ref {claim:onebis}). Since the fibre of $x$ in $\pi$ is $0$--dimensional, also the fibre of $\pi_{\vert \Gamma}$ is finite. This implies that $\langle \Gamma\rangle \cap \Pi_k$ is only one point $z\in \Gamma$ and that the line $\langle x,z\rangle$ is not contained in $\Gamma$. So the line $\langle x,z\rangle$ meets $\Gamma$ only at $z$ and $x$, contradicting
$\pi(x)=\pi(y)$. \end{proof}

Let $H$ be the hyperplane in $\p^ {\frac {k+2}2 n+k+1}$  which is tangent to $X$ along $\Gamma_k$ and let $H'$ be its image via $\pi$.

\begin {claim}\label {claim:fourbis} The inverse of the map $\pi: X\map \p^ n$ is well defined off $H'$.\end{claim}

\begin{proof}[Proof of the Claim]  We resolve the indeterminacies of $\pi$ by bowing-up
$\Gamma_k$. If $f: \tilde X\to X$ is this blow--up, then $p=\pi\circ f: \tilde X \to \p^ n$ is a morphism. Note that $\tilde X$ is smooth along the exceptional divisor $E$ by Claim  \ref {claim:alf}. Points on the exceptional divisor $E$ are mapped via $p$ to points of $H'$.

Let $z\in \p^ n$ be a point where the inverse of $\pi$ is not defined.
If $x\in X$ is a point such that $\pi(x)=z$, then the subspace $\Pi_x:=\langle \Pi_k,x\rangle$ intersects $X$ along an irreducible positive dimensional subvariety $Z$ containing $x$, which is contracted to a point via $\pi$. Since $\Pi_k\cap X=\Gamma_k$ (see Corollary \ref {lem:inter}), then $Z\cap \Gamma_k\neq \emptyset$. Let $Z'$ be the strict transform of $Z$ on $\tilde X$. Then $Z'$ intersects $E$, and is contracted to a point by $p$. Hence $z\in H'$. \end{proof}

Now we are able to finish the proof of the theorem. By Claim \ref {claim:fourbis}, $X-(H\cap X)$ is isomorphic to the affine space $\p^ n-H'$. It remains to prove that there is no point $x\in X$ contained in
all hyperplanes tangent to the tangential $k$--contact loci $\Gamma_k$. Suppose that such a point exists. Let $p_1,\ldots,p_k\in X$ be general points and consider again the tangential projection $\tau_k$ from $T_{X,p_1,\ldots,p_k}$. Since $S^ {k-1}(X)$ is not a cone (see Claim \ref {claim:tap}), $\tau_k$ is well defined at $p$. Let $z$ be its image via $\tau_k$. Then all tangent hyperplanes to $X_k$ would contain $z$, hence $X_k$ would be a cone, a contradiction. \end{proof}

We can now prove the classification theorem. By taking into account Theorem \ref {thm:ksev}
and Zak's Classification Theorem \ref {thm:zakclassthm}, we may consider only $k$--Severi varieties with $k\geq 2$.

  \begin {theorem}\label{thm:classthm}
 Let $X\subseteq \p^ r$ be a $k$--Severi variety, with $k\geq 2$. Then $X$ is one of the following varieties

 \begin {itemize}

\item [(i)] the $(k+1)$--dimensional Veronese variety $V_{2,k+1}$ in $\p^ {\frac {k(k+3)}2}$;

\item [(ii)] the $2(k+1)$--dimensional Segre variety ${\rm Seg}(k+1,k+1)$ in $\p^ {k^ 2+4k+3}$;

\item [(iii)] the $4(k+1)$--dimensional Grassmann variety $\G(1,2k+3)$ in $\p^ { {{2k+4}\choose 2}-1}$.

 \end{itemize}

 \end{theorem}

\begin{proof}   The varieties in (i)--(iii) are $R_k$--varieties (see Example \ref {ex:veretc}).
Moreover they are $k$--Severi varieties, i.e. \eqref {eq:boundsev} holds for them (see
\cite {Zak}).

Set now $s=s^{(k)}+1=\frac {k+2}2 n + k+1$ and take a general
 projection $X'$ to $\p^ s$. Then $X'$ is still a $k$--Severi variety.
 Moreover we have
 $k_0(X')=k+1$, $\psi_i(X')= \psi_i(X)=if$, $f_i(X')=f_i(X)=\frac {i(i+1) }2f$ for all $i=1,\ldots,k$ and $n=(k+1)f$ by Lemma \ref {step2}. Thus $X'$ is a Scorza variety. The classification follows from Zak's classification of Scorza varieties, which implies that Scorza varieties are linearly normal, in particular $X=X'$. Notice the smoothness of the entry loci, which follows by Lemma \ref {cor:all}  and Claims \ref  {claim:apt} and \ref {claim:alf}. This is essential in Zak's argument (see Remark \ref {rem:entry1}).  \end{proof}

 \begin{remark}\label {rem:ksevsev}
  The $k$--Severi varieties, with $k\geq 2$, are related to the unitary composition algebras
 $\R, \C, \Ha$. In this case we take all $(k+1)\times (k+1)$
 hermitian matrices  $A$ and impose that $\rk(A)=1$. This gives the equations
 defining the $k$--Severi varieties. Again the $k$--secant variety to a $k$--Severi
 variety is defined by the vanishing of $\det (A)$. The absence of the analogue
 of the $E_6$--variety, related to the composition algebra $\Oa$, reflects the
 absence of higher order determinants on $\Oa$ (see Remark \ref {rem:sevsev}).

 A quick proof of the classification of Severi varieties can be obtained by using
 the beautiful ideas contained in \cite {ru}. On the same lines one can give a proof of the classification of
 $k$--Severi varieties, alternative to the one
 described above based on Zak's classification of Scorza varieties. We do not dwell on this here.
  \end{remark}

 \section{Speculations}\label {sec:spec}

 Perhaps the main motivation for Zak's beautiful piece of work was the following
  famous conjecture by Hartshorne (see \cite {???}):

  \begin{conjecture}\label {conj:harta}  [Hartshorne's conjecture]
Let $X\subset \p^ r$ be a smooth, irreducible,
non--degenerate variety of dimension $n$. If $3n>2r$ then $X$ is a complete intersection
of $r-n$ hypersurfaces in $\p^ r$. \end{conjecture}

This in turn was motivated by Barth--Larsen's fundamental result (see \cite {bl}) to the effect that
smooth varieties $X\subset \p^ r$ of \emph {low codimension} are \emph {topologically similar}
to $\p^ r$. Barth--Larsen's theorem, in our context, can be stated as follows:

\begin{theorem}\label {thm:bl} Let $X\subset \p^ r$ be a smooth, irreducible variety. Then for any non--negative integer $i<f_1(X)$ the natural map

$$\rho_{X,i}:H^ i(\p^ r,\Z)\to H^ i(X,\Z)$$
\noindent is an isomorphism. In particular, if  $f_1(X)\geq 2$ then $X$ is simply connected and if $f_1(X)\geq 3$ then ${\rm Pic}(X)$ is generated by the hyperplane class.
\end{theorem}

One of the basic steps in the proof of
Conjecture \ref  {conj:harta} would be to show that, under
the hypotheses, $X$ is projectively Cohen--Macaulay. Since linear normality is the first
na\"ive requirement for being projectively Cohen--Macaulay, this is the motivation for
Conjecture \ref {conj:hartb}, which in turn motivates Zak's theorems.

Now, in presence of our refined form of Zak's linear normality theorem, one may speculate
on the possibility of having an even more general view on Hartshorne's conjecture. This
is what we want to present next. To be precise, we want to propose the following:

  \begin{conjecture}\label {conj:harte}  [Extended Hartshorne's conjecture]
  There is a suitable function $f(r,n,k)$ such that the following happens.
Let $X\subset \p^ r$ be a $k$--smooth, irreducible,
non--degenerate variety of dimension $n$. If $(k+2)n>2r$ then  $\II_X$ is
generated by at most $f(r,n,k)$ elements.\end{conjecture}

This conjecture does not make too much sense unless one specifies the form
of the function $f(r,n,k)$. What we intend, is that $f(r,n,k)$ should be \emph {reasonably
small}. If one wants to be really bold, one may even conjecture that $f(r,n,k)=k(r-n)$.
A further strengthening of the conjecture would be to replace the ideal sheaf $\II_X$
with the homogeneous ideal $I_X$.

\begin{example}\label {ex:motiv} There are varieties \emph {at the boundary} of Hartshorne's
conjecture. One of them is $\G(1,4)$, which has dimension 6 in $\p^ 9$. Its homogeneous
ideal is generated by 5 quadrics. This would fit with Conjecture \ref {conj:harte} for
$k=2$ and $f(r,n,2)=2(r-n)$, but unfortunately $\G(1,4)$ is not $2$--smooth, since it contains lines.

Another variety at the boundary of Hartshorne's
conjecture is the 10--dimensional \emph {spinor variety} $S_4\subset \p^ {15}$.
 Its homogeneous
ideal being generated by 10 quadrics. Again this would fit with Conjecture \ref {conj:harte} for $k=2$ and $f(r,n,2)=2(r-n)$, but this variety is also not $2$--smooth.

More varieties at the boundary of Hartshorne's conjecture, actually all the known ones, are deduced from these by
pulling them back via a general morphism $\p^ r\to \p^ r$. Now these can in general be $2$--smooth and
Conjecture \ref {conj:harte} holds  for them with $f(r,n,2)=2(r-n)$.

The examples of $\G(1,4)$ and $S_4$ suggest that the $k$--smoothness assumption in Conjecture
\ref {conj:harte}  might even be too strong. May be something like the $R_k$-- property could suffice.
\end{example}

\begin {remark} At this point a related natural question arises: is there, in this same spirit, any extension of Barth--Larsen's Theorem \ref {thm:bl}? By taking into account \eqref {eq:scorza} we see that

$$2f_k\geq {k(k+1)}f_1$$

\noindent and therefore one might ask: is the map $\rho_{X,i}$ an isomorphism for
all positive integers $i$ such that $k(k+1)(i+1)\leq 2f_k(X)$, under the assumption that $X\subset \p^ r$ be a $k$--smooth, irreducible variety? Or, is $\rho_{X,i}$ an isomorphism under the condition

$${{k+1}\choose 2} i\leq (k+1)n-r- {k\choose 2}$$
if $X\subset \p^ r$ is a $k$--smooth, irreducible variety?
\end{remark}

\end{document}